\documentclass[reqno, 11pt]{amsart}

\usepackage{amsmath,amsthm,amssymb}
\usepackage{booktabs}
\usepackage{xcolor}
\usepackage{geometry}
\usepackage{hyperref}
\hypersetup{
    colorlinks,
    linkcolor={red!60!black},
    citecolor={green!60!black},
    urlcolor={blue!60!black}
}

\theoremstyle{plain}
\newtheorem{Theorem}{Theorem}[section]
\newtheorem{Proposition}[Theorem]{Proposition}
\newtheorem{Lemma}[Theorem]{Lemma}
\newtheorem{Remark}[Theorem]{Remark}

\theoremstyle{definition}
\newtheorem{Definition}[Theorem]{Definition}

\DeclareMathOperator{\SL}{SL}
\DeclareMathOperator{\GL}{GL}
\DeclareMathOperator{\Sym}{Sym}
\DeclareMathOperator{\tr}{tr}
\newcommand{\F}{\mathbb F}
\newcommand{\Z}{\mathbb Z}

\renewcommand{\le}{\leqslant}
\renewcommand{\ge}{\geqslant}

\title{There are no sharply transitive subsets of $\SL(2,q)$ for $q\ge 13$}

\author{John Bamberg}
\address[JB]{Department of Mathematics and Statistics, 
The University of Western Australia, 
35 Stirling Highway, 
Perth, W.A. 6009, Australia.}
\email{John.Bamberg@uwa.edu.au}

\author{Sam Mattheus}
\address[SM]{Department of Mathematics and Data Science, 
Vrije Universiteit Brussel, 
Pleinlaan 2, 1050 Brussels, Belgium.}
\email{Sam.Mattheus@vub.be}

\date{\today}


\begin{document}

\begin{abstract}
It was known at least to L.E.~Dickson in 1901 that $\mathrm{SL}(2,q)$, in its
natural action on $\F_q^2\setminus\{0\}$, has
a sharply transitive subgroup only when $q\in\{2,3,5,7,11\}$. For $q$ prime, this result stems from
Galois' letter to Chevalier in 1832. We extend this result to sharply transitive \emph{subsets} of 
$\mathrm{SL}(2,q)$ and show that they only exist when $q\in\{2,3,5,7,11\}$.
\end{abstract}

\maketitle

\section{Introduction}

Let $G$ be a transitive permutation group on a finite set $\Omega$.  A subset
$S\subseteq G$ is called \emph{sharply transitive} if, for every ordered pair
$(\alpha,\beta)\in\Omega\times\Omega$, there is a unique element $s\in S$
such that $\alpha^s=\beta$.  In this paper we consider the natural action of
$G=\SL(2,q)$ on $V^\#=\F_q^2\setminus\{0\}$.
This action has degree $q^2-1$, and we ask whether $G$ contains a sharply
transitive subset on $V^\#$. 

A closely related result was obtained by Bonisoli and Korchmaros \cite{BonisoliKorchmaros}, building
on the work of Bonisoli \cite{Bonisoli}. They proved that, for every prime power $q$, 
every sharply transitive subset of $\mathrm{PGL}(2,q)$, in its natural action on the projective line $\mathrm{PG}(1,q)$, 
is a coset of a subgroup. A simpler proof has appeared recently in \cite{Eberhard}. Other interesting actions
of finite transitive permutation groups were addressed by M\"uller and Nagy \cite{MullerNagy}.
It is well-known via the correspondence with \emph{all-even} Latin squares that the natural actions of 
the alternating groups $A_n$ always have sharply transitive subsets for $n\ge 3$ (see \cite{Donovan_et_al}).
Ernst and Schmidt studied sharply transitive subsets of $\mathrm{GL}(n,q)$ acting on certain types of flags \cite{ErnstSchmidt}.

The subgroup case is classical. Galois already identified the exceptional prime cases $p=5,7,11$ in his 1832 letter to Chevalier, and Dickson's classification \cite[Chapter XII]{Dickson} of the subgroups of $\mathrm{PSL}(2,q)$ shows that a sharply transitive subgroup on $V^\#$ exists precisely for $q=2,3,5,7,11$. In these cases one obtains, respectively,
$C_3$, $Q_8$, $2.A_4$, $2.S_4$, $2.A_5$ of order $q^2-1$. 
These examples belong to the classical theory of finite transitive
linear groups.  After choosing a nonzero vector to serve
as identity, a sharply transitive subgroup of $\GL(V)$ on $V^\#$ is the
multiplicative group of a finite nearfield on the additive group of
$V$; in particular, the cases $q=5,7,11$ occur among the exceptional
nearfields of Zassenhaus \cite[Section~5 and Table~1]{Nagy}.
The corresponding problem considered here has a rather different character: a sharply transitive set in $\SL(2,q)$ need not be a subgroup, and its existence is instead constrained by the finer combinatorial and representation-theoretic structure of $\SL(2,q)$.

The problem is also naturally connected to certain objects in finite classical polar spaces.
For $q$ odd, maximal partial ovoids of size $q^2-1$ in the parabolic quadric
$Q(4,q)$ may be represented as sharply transitive subsets of $\SL(2,q)$.
This connection was made explicit by Coolsaet, De Beule and Siciliano, who
gave a uniform description of all known examples and related them to spread
sets~\cite{CoolsaetDeBeuleSiciliano}.  Earlier work of De Beule and G\'{a}cs \cite{DeBeuleGacs},
and subsequently De Beule \cite{DeBeule}, ruled out such maximal partial ovoids when
$q=p^h$ and $h>1$. Thus the remaining case is $q=p$ prime.
Our main theorem proves that there are no further examples than what is known.

\begin{Theorem}\label{maintheorem}
Let $p\ge 13$ be a prime.  Then $\SL(2,p)$, in its natural action on
$\F_p^2\setminus\{0\}$, contains no sharply transitive subset.
\end{Theorem}

Consequently, sharply transitive subsets of $\SL(2,q)$ in its
natural action on $\F_q^2\setminus\{0\}$ exist if and only if
$q\in\{2,3,5,7,11\}$. 
Equivalently, the corresponding maximal partial ovoids
of $Q(4,q)$ of size $q^2-1$ are exactly the known examples, settling the conjecture raised in \cite{CoolsaetDeBeuleSiciliano} in the affirmative.

\subsection*{Declaration of the use of AI}
Artificial intelligence played a substantive role in the development of the proof in this paper, and we wish to describe that role as precisely as possible. At an early stage of the project, we supplied ChatGPT (then using GPT-5.5) with a substantial collection of our ideas and calculations based on the association scheme approach to the problem. In exploring and extending these arguments, ChatGPT identified the modular representation-theoretic viewpoint that ultimately led to the present proof. Some of the intermediate results produced in this process were close to arguments that we had already obtained, or could have obtained routinely, from the association scheme. In particular, Lemma \ref{lem:outer-distribution} is essentially the outer-distribution calculation in that scheme. 
Lemma \ref{lem:intersection-numbers} is a standard computation of certain intersection numbers in the relevant association scheme, while the count of two-step walks in Proposition \ref{prop:two-step-walks} is a finer version of calculations we had made.
We therefore regard these parts of the argument principally as refinements and translations of ideas already present in our work.

The genuinely new step suggested by ChatGPT was Theorem \ref{thm:nonsplit-residue-general}: the congruence condition governing the nonsplit values. This observation was not present in our earlier approach and proved to be the decisive bridge from the association scheme calculations to the modular argument. More significantly still, the central mechanism of the final proof -- in particular, the introduction of the function $F_Q(z)$ and the argument built around it -- was discovered by ChatGPT. This was not a reformulation of an argument previously known to us, but a genuinely new proof strategy proposed during our interaction with the model. We subsequently checked the argument in detail, supplied the necessary mathematical justification, and developed it into the form presented here.

Finally, we mention that the authors had applied \emph{Delsarte-theory}, to the ultimate end for small examples (e.g., $q=13$) to no avail. The LP-bound, outer distributions, and other tools from ordinary representation theory seemed to not disallow the existence of a sharply transitive subset of $\SL(2,q)$, and it was the intervention of modular representation theory that cracked the problem open: something that, to the best of our knowledge, has rarely been seen before in the literature.

\section{Trace relations and elementary counts}

For the remainder of the main proof, let $p$ be an odd prime and put
$G=\SL(2,p)$.  Let $\chi$ denote the quadratic character of $\F_p$, extended
by $\chi(0)=0$, and put
\[
\begin{aligned}
        \Sigma_{\mathrm{sp}}
        &=\{t\in\F_p:\chi(t^2-4)=1\},\\
        \Sigma_{\mathrm{ns}}
        &=\{t\in\F_p:\chi(t^2-4)=-1\}.
\end{aligned}
\]

The abbreviations here stand for `split' and `nonsplit' trace values, in regard to the possible eigenvalues of
the given matrix. We will say that a matrix has \textbf{split} trace value if it has two eigenvalues in $\F_p$,
and \textbf{nonsplit} if those two eigenvalues lie in $\F_{p^2}\backslash\F_p$.
Indeed, if $g\in G$ has trace
$t$, then its characteristic polynomial is
\[
        X^2-tX+1,
\]
with discriminant $t^2-4$.  Hence $t\in\Sigma_{\mathrm{sp}}$ exactly when $g$
has two distinct eigenvalues in $\F_p$, while $t\in\Sigma_{\mathrm{ns}}$
exactly when its eigenvalues lie in $\F_{p^2}\setminus\F_p$.  The excluded
values $t=\pm2$ are precisely the repeated-root cases.

We briefly recall the conjugacy classes of $G=\SL(2,p)$, following
Humphreys \cite{Humphreys1975}. Fix a generator $\nu$ of $\F_p^\ast$, and
put
\[
 z=-I,\qquad
 a=\begin{pmatrix}\nu&0\\0&\nu^{-1}\end{pmatrix},\qquad
 c=\begin{pmatrix}1&1\\0&1\end{pmatrix},\qquad
 d=\begin{pmatrix}1&\nu\\0&1\end{pmatrix}.
\]
Let $b$ be an element of order $p+1$ which is not diagonalisable over
$\F_p$.  Then the $p+4$ conjugacy classes of $G$ have representatives
\[
  1,\quad z,\quad
  a^\ell\ \left(1\le \ell\le\frac{p-3}{2}\right),\quad
  b^m\ \left(1\le m\le\frac{p-1}{2}\right),\quad
  c,\ d,\ zc,\ zd.
\]
The classes represented by $1,z,a^\ell,b^m$ are precisely the
$p$-regular classes in the terminology of Humphreys. Thus the $a^\ell$ are
the split regular semisimple elements, having two distinct eigenvalues in
$\F_p$, whereas the $b^m$ are the nonsplit regular semisimple elements,
whose eigenvalues lie in $\F_{p^2}\setminus\F_p$. Their conjugacy classes
have sizes respectively $p(p+1)$ and $p(p-1)$. The two nonidentity
unipotent classes are represented by $c,d$, and the two classes with trace
$-2$ by $zc,zd$; each of these four classes has size $(p^2-1)/2$.
Consequently, the regular trace values split into the two families
$\Sigma_{\mathrm{sp}}$ and $\Sigma_{\mathrm{ns}}$, while $\pm2$ are the two
nonregular trace values.
In particular, the split values correspond to the
classes of the $a^\ell$, and the nonsplit values to the classes of the
$b^m$. Indeed, 
\begin{equation}
        |\Sigma_{\mathrm{sp}}|=\frac{p-3}{2},
        \qquad
        |\Sigma_{\mathrm{ns}}|=\frac{p-1}{2}.
        \label{eq:split-nonsplit-counts}
\end{equation}

\begin{Definition}
\label{def:trace-relations}
For $x,y\in G$, define
\[
\begin{aligned}
        x\,R_\delta\,y&\iff x=y,\\
        x\,R_\infty\,y&\iff y=-x,\\
        x\,R_t\,y&\iff y\neq\pm x\text{ and }\tr(x^{-1}y)=t
        \qquad(t\in\F_p).
\end{aligned}
\]
If $S\subseteq G$ is sharply transitive, put
\[
        c_i(x)=|\{s\in S:x\,R_i\,s\}|,
\]
and write
\[
        \epsilon_x=c_\delta(x)=\mathbf 1_S(x),
        \qquad
        \alpha_x=c_\infty(x)=\mathbf 1_S(-x).
\]
\end{Definition}

The trace relations form a symmetric association scheme.  Indeed,
$\tr(g^{-1})=\tr(g)$ for $g\in\SL(2,p)$, so each $R_t$ is symmetric.  The
expert reader with a background in association schemes will recognise the
$c_i$ as values of the \emph{outer distribution} of $S$.

\begin{Lemma}
\label{lem:outer-distribution}
Let $S\subseteq G$ be sharply transitive on $\F_p^2\setminus\{0\}$. Then,
for every $x\in G$, $c_2(x)=(p+1)(1-\epsilon_x)$, $c_{-2}(x)=(p+1)(1-\alpha_x)$, and $c_t(x)=p+1$ for all $t\in\Sigma_{\mathrm{sp}}$.
\end{Lemma}

\begin{proof}
Fix $x\in G$ and $\lambda\in\F_p^\ast$. Count the set
\[
\mathcal P_\lambda(x)
=
\{(v,s)\in(\F_p^2\setminus\{0\})\times S:
  sv=\lambda xv\}.
\]
For every nonzero $v$, sharp transitivity gives a unique $s\in S$ with
$sv=\lambda xv$, so
\begin{equation}
        |\mathcal P_\lambda(x)|=p^2-1.
        \label{eq:eigenvector-count-total}
\end{equation}

Take first $\lambda=1$. The element $s=x$, when it belongs to $S$,
contributes all $p^2-1$ nonzero vectors. Every other $s$ for which
$x^{-1}s$ has eigenvalue $1$ is a nonidentity unipotent element of trace
$2$ and contributes the $p-1$ nonzero vectors in its one-dimensional fixed
space. These are exactly the elements counted by $c_2(x)$. Hence
\[
        p^2-1=\epsilon_x(p^2-1)+c_2(x)(p-1),
\]
which gives the formula for $c_2(x)$.

The same argument with $\lambda=-1$ shows that $s=-x$ contributes
$p^2-1$ vectors, while each noncentral element of trace $-2$ contributes
$p-1$ vectors. Thus
\[
        p^2-1=\alpha_x(p^2-1)+c_{-2}(x)(p-1).
\]

Finally, let $\lambda\neq\pm1$ and put $t=\lambda+\lambda^{-1}$. Then
\[
        t^2-4=(\lambda-\lambda^{-1})^2
\]
is a nonzero square. Every element $x^{-1}s$ of trace $t$ has the distinct
eigenvalues $\lambda$ and $\lambda^{-1}$, and therefore contributes exactly
$p-1$ nonzero $\lambda$-eigenvectors to $\mathcal P_\lambda(x)$. Conversely,
every $t\in\Sigma_{\mathrm{sp}}$ arises in this way. Thus
\eqref{eq:eigenvector-count-total} gives
\[
        p^2-1=c_t(x)(p-1),
\]
and hence $c_t(x)=p+1$.
\end{proof}

The following is a basic fact in the theory of finite fields
\cite[Theorem 5.48]{finitefields}.

\begin{Lemma}
\label{lem:quadratic-character-sum}
For every $z\in\F_p^\ast$, we have
$\sum_{t\in\F_p}\chi(t^2-z)=-1$.
\end{Lemma}

The rest of this section is devoted to giving a precise
equation for $\sum_{y\in G}c_a(y)^2$, which corresponds to counting
two-step walks (see Proposition
\ref{prop:two-step-walks}). 
First we compute, in the language of association schemes, the intersection numbers $p_{aa}^h$.

\begin{Lemma}
\label{lem:intersection-numbers}
Let $a\in\F_p$ satisfy $\chi(a^2-4)=-1$ and $a\ne0$, and put $b:=a^2-2$.
For $h\in G$, define
\[
        N_a(h):=
        |\{y\in G:\tr(y)=a,\ \tr(y^{-1}h)=a\}|.
\]
Then $b^2-4$ is a nonsquare, and the values of $N_a(h)$ are listed in
Table \ref{tab:intersection-numbers}.
\begin{table}[!h]
\centering
\begin{tabular}{ll}
\toprule
Condition on $h$&$N_a(h)$\\
\midrule
$h=I$&$p(p-1)$\\
$h=-I$&$0$\\
$h\ne I$, $\tr(h)=2$&$0$\\
$h\ne-I$, $\tr(h)=-2$&$p$\\
$\tr(h)\in\Sigma_{\mathrm{sp}}$&$p-1$\\
$\tr(h)=b$&$1$\\
$\tr(h)\in\Sigma_{\mathrm{ns}}\setminus\{b\}$&$p+1$.\\
\bottomrule
\end{tabular}\smallskip
\caption{Values of $N_a(h)$ depending on $h$.}
\label{tab:intersection-numbers}
\vspace{-2em}
\end{table}
\end{Lemma}

\begin{proof}
We have $b^2-4=a^2(a^2-4)$, which is a nonsquare. 
Note that the number $N_a(h)$ is invariant under conjugating $h$ in $\GL(2,p)$.
Write
\[
        y=\begin{pmatrix}u&v\\w&a-u\end{pmatrix}.
\]

First suppose $h=I$. 
The determinant equation $\det(y)=1$ is
\[
        vw=u(a-u)-1.
\]
The right-hand side never vanishes, since $u^2-au+1$ has the nonsquare
discriminant $a^2-4$. Thus for each of the $p$ choices of $u$ there are
$p-1$ choices for $(v,w)$, and $N_a(I)=p(p-1)$. If $h=-I$, the second
trace condition reads $-a=a$, which is impossible because $a\ne0$.

Now let
\[
        h_\sigma=\begin{pmatrix}\sigma&t\\0&\sigma\end{pmatrix},
        \qquad t\ne0,
        \qquad \sigma\in\{1,-1\}.
\]
These matrices represent the noncentral trace-$2\sigma$ cases. A direct
calculation gives
\[
        \tr(y^{-1}h_\sigma)=\sigma a-tw.
\]
For $\sigma=1$, the second trace condition forces $w=0$, after which
$\det(y)=1$ becomes $u^2-au+1=0$, which has no solution. For $\sigma=-1$,
the condition forces $w=-2a/t\ne0$; for each $u\in\F_p$, the determinant
equation then determines $v$ uniquely. This gives the third and fourth rows
of Table \ref{tab:intersection-numbers}.

It remains to consider a regular semisimple element $h$ of trace
$\tau\ne\pm2$. Since $h$ is not a scalar matrix, choose $v_0$ such that
$\{v_0,hv_0\}$ is a basis. By Cayley--Hamilton,
$h^2-\tau h+I=0$, so with respect to this basis
\[
        h=\begin{pmatrix}0&-1\\1&\tau\end{pmatrix}.
\]
Now the condition $\tr(y^{-1}h)=a$ gives
$w=a+v-\tau u$. Substituting into $\det(y)=1$ yields
\begin{equation}
        u^2+v^2-\tau uv-au+av+1=0.
        \label{eq:affine-conic-before-translation}
\end{equation}
The translation
\[
        u=U+\frac{a}{\tau+2},
        \qquad
        v=V-\frac{a}{\tau+2}
\]
turns \eqref{eq:affine-conic-before-translation} into
\begin{equation}
        U^2+V^2-\tau UV
        =\frac{a^2-\tau-2}{\tau+2}.
        \label{eq:affine-conic-normal-form}
\end{equation}
Put $\Delta=\tau^2-4$. For a nonzero right-hand side $c_0$, the number of
solutions of
$U^2+V^2-\tau UV=c_0$ is $p-\chi(\Delta)$. Indeed,
\[
        4(U^2+V^2-\tau UV)=(2U-\tau V)^2-\Delta V^2,
\]
and Lemma~\ref{lem:quadratic-character-sum} gives
\[
\begin{aligned}
&\left|\left\{(U,V):(2U-\tau V)^2-\Delta V^2=4c_0\right\}\right|\\
&\hspace{3em}=
\sum_{V\in\F_p}\bigl(1+\chi(\Delta V^2+4c_0)\bigr)
=p-\chi(\Delta).
\end{aligned}
\]
Thus the answer is $p-1$ when $\tau\in\Sigma_{\mathrm{sp}}$ and $p+1$
when $\tau\in\Sigma_{\mathrm{ns}}$. The right-hand side of
\eqref{eq:affine-conic-normal-form} vanishes exactly when
$\tau=a^2-2=b$. In this case
\[
        \Delta=b^2-4=a^2(a^2-4)
\]
is a nonsquare. Hence the homogeneous binary quadratic form is anisotropic
and has only the zero solution. This gives $N_a(h)=1$ at the exceptional
trace value $b$ and completes the table.
\end{proof}

\begin{Proposition}
\label{prop:two-step-walks}
Under the hypotheses of Lemma~\ref{lem:intersection-numbers}, if
$S\subseteq G$ is sharply transitive, then
\[
        \sum_{y\in G}c_a(y)^2
        =p(p^2+p-2)|S|-p\sum_{x\in S}c_b(x).
\]
\end{Proposition}

\begin{proof}
For $x\in G$, let
\[
\mathcal M_a(x):=
\{(y,s)\in G\times S:
  \tr(x^{-1}y)=a,\ \tr(y^{-1}s)=a\}.
\]
For a fixed $s\in S$, left multiplication by $x^{-1}$ shows that the
number of possible $y$ is $N_a(x^{-1}s)$. Applying Table
\ref{tab:intersection-numbers} according to the trace relation between
$x$ and $s$ gives
\begin{align*}
|\mathcal M_a(x)|
={}&p(p-1)\epsilon_x+p\,c_{-2}(x)
 +(p-1)\sum_{t\in\Sigma_{\mathrm{sp}}}c_t(x)
 +(p+1)\sum_{t\in\Sigma_{\mathrm{ns}}}c_t(x)
 -p\,c_b(x).
\end{align*}
The final term is the deficit of $p$ at the exceptional trace value
$b$: there is one midpoint there, rather than the generic $p+1$.
By Lemma~\ref{lem:outer-distribution} and
\eqref{eq:split-nonsplit-counts},
\[
        c_2(x)=(p+1)(1-\epsilon_x),
        \qquad
        c_{-2}(x)=(p+1)(1-\alpha_x),
\]
and
\[
        \sum_{t\in\Sigma_{\mathrm{sp}}}c_t(x)
        =\frac{p-3}{2}(p+1).
\]
Using also
\[
|S|=\epsilon_x+\alpha_x+c_2(x)+c_{-2}(x)
 +\sum_{t\in\Sigma_{\mathrm{sp}}}c_t(x)
 +\sum_{t\in\Sigma_{\mathrm{ns}}}c_t(x),
\]
a straightforward simplification gives
\begin{equation}
        |\mathcal M_a(x)|
        =2p^2\epsilon_x+p(p^2-p-2)-p\,c_b(x).
        \label{eq:pointwise-two-walk-count}
\end{equation}

Now sum over $x\in S$. For a fixed midpoint $y$, there are $c_a(y)$
choices for the first endpoint $x\in S$ and independently $c_a(y)$ choices
for the second endpoint $s\in S$. Hence
\[
        \sum_{x\in S}|\mathcal M_a(x)|
        =\sum_{y\in G}c_a(y)^2.
\]
Summing \eqref{eq:pointwise-two-walk-count} over $x\in S$ gives
\[
        \sum_{y\in G}c_a(y)^2
        =p(p^2+p-2)|S|-p\sum_{x\in S}c_b(x),
\]
as required.
\end{proof}

\section{A modular congruence at nonsplit trace values}

The elementary count in Lemma~\ref{lem:outer-distribution} determines the
local distribution at the two nonregular trace values $\pm2$ and at every
regular split trace value. It does not determine the individual coordinates
with trace in $\Sigma_{\mathrm{ns}}$. The important result in this section is
Theorem~\ref{thm:nonsplit-residue-general}, which gives their residues
modulo $p$.

Let $k$ be the algebraic closure of $\F_p$, let $W=k^2$ be the natural
$kG$-module, and, for $d\ge0$, put
\[
        W_d:=\Sym^d(W).
\]
We realise $W_d$ as the space of homogeneous polynomials of degree $d$ in
two variables, and write $\rho_d:G\to\GL(W_d)$ for the induced
representation. The Dickson polynomials of the second kind are the
polynomials $E_d(T,c)\in\Z[T,c]$ defined by
\[
        E_0(T,c)=1,
        \qquad
        E_1(T,c)=T,
        \qquad
        E_d(T,c)=TE_{d-1}(T,c)-cE_{d-2}(T,c).
\]
In particular,
\[
E_d(\lambda+\lambda^{-1},1)
=\lambda^d+\lambda^{d-2}+\cdots+\lambda^{-d}.
\]
If $g\in\GL(2,k)$ has eigenvalues
$\lambda,\mu$, then its eigenvalues on $W_d$ are 
\[
\lambda^d,\lambda^{d-1}\mu,\ldots,\lambda\mu^{d-1},\mu^d.
\]
Consequently, $\tr\rho_d(g)=E_d(\tr(g),\det(g))$, and, for
$g\in\SL(2,p)$, we have $\tr\rho_d(g)=E_d(\tr(g),1)$.

\begin{Remark}
The irreducible $p$-modular representations of $\SL(2,p)$ are precisely the
$\rho_d$, for $0\le d\le p-1$ (cf. \cite[\S5.8]{Benson} and 
\cite{Srinivasan}). The calculation above shows that their trace
functions are $g\mapsto E_d(\tr(g),1)$. 
\end{Remark}

\begin{Lemma}\label{lem:symmetric-power-average-vanishes}
Suppose that $S\subseteq\SL(2,p)$ is sharply transitive on
$\F_p^2\setminus\{0\}$. Then, for $1\le d\le p-2$,
\[
        \sum_{s\in S}\rho_d(s)=0.
\]
\end{Lemma}

\begin{proof}
Put $M_d=\sum_{s\in S}\rho_d(s)$. It is enough to evaluate $M_d$ on pure
powers, since the pure powers $v^d$, with $v\in\F_p^2$, span $W_d$ when
$d<p$. Indeed, choosing $d+1$ distinct values $u\in\F_p$, the vectors
$(X+uY)^d$ form a Vandermonde system in the monomial basis of $W_d$.
Let $v\in\F_p^2\setminus\{0\}$. Sharp transitivity gives
$\{sv:s\in S\}=\F_p^2\setminus\{0\}$, and hence
\[
        M_d(v^d)=\sum_{w\in\F_p^2\setminus\{0\}}w^d.
\]
Grouping the nonzero vectors by one-dimensional subspaces, the contribution
of a line spanned by $u$ is
\[
        \sum_{\lambda\in\F_p^\ast}(\lambda u)^d
        =\left(\sum_{\lambda\in\F_p^\ast}\lambda^d\right)u^d=0,
\]
because $1\le d\le p-2$. Thus $M_d=0$.
\end{proof}

Applying Lemma~\ref{lem:symmetric-power-average-vanishes} to $x^{-1}S$ and
taking traces gives
\begin{equation}
        \sum_{s\in S}E_d\bigl(\tr(x^{-1}s),1\bigr)=0
        \qquad(1\le d\le p-2)
        \label{eq:dickson-trace-identity}
\end{equation}
in $\F_p$, for every $x\in G$.

\begin{Theorem}
\label{thm:nonsplit-residue-general}
Suppose that $S\subseteq\SL(2,p)$ is sharply transitive on
$\F_p^2\setminus\{0\}$. If $t^2-4$ is a nonsquare in $\F_p$, then, for
every $x\in G$,
\[
        c_t(x)\equiv t^2-3\pmod p.
\]
\end{Theorem}

\begin{proof}
Fix $x\in G$. For $t\in\F_p$, define $C_t(x):=|\{s\in S:\tr(x^{-1}s)=t\}|$,
so that $C_t(x)=c_t(x)$ for $t\ne\pm2$,
while $C_2(x)=c_2(x)+\epsilon_x$ and $C_{-2}(x)=c_{-2}(x)+\alpha_x$.
By Lemma~\ref{lem:outer-distribution},
$C_t(x)\equiv1\pmod p$ whenever $t^2-4$ is a nonzero square, and also
\[
        C_2(x)\equiv C_{-2}(x)\equiv1\pmod p.
\]
Equation \eqref{eq:dickson-trace-identity} gives
\[
        \sum_{t\in\F_p}C_t(x)E_d(t,1)=0
        \qquad(1\le d\le p-2),
\]
and the total-size equation is
\begin{equation}
        \sum_{t\in\F_p}C_t(x)=|S|=p^2-1\equiv-1\pmod p.
        \label{eq:trace-total-size}
\end{equation}

We claim that the unique solution modulo $p$ compatible with these equations
and the known split and nonregular values is
\begin{equation}
        C_t(x)\equiv
        \begin{cases}
        1,&t^2-4\text{ is a square or zero},\\
        t^2-3,&t^2-4\text{ is a nonsquare}
        \end{cases}
        \pmod p.
        \label{eq:claimed-trace-residues}
\end{equation}

Let $\mathcal U:=\{t\in\F_p:t^2-4\text{ is a nonsquare}\}$.
By Lemma~\ref{lem:quadratic-character-sum},
$|\mathcal U|=(p-1)/2$. If two solutions agree off $\mathcal U$, their
difference $(u_t)_{t\in\mathcal U}$ satisfies
\[
        \sum_{t\in\mathcal U}u_tE_d(t,1)=0
        \qquad
        \left(0\le d\le\frac{p-3}{2}\right),
\]
where $d=0$ is supplied by \eqref{eq:trace-total-size}. The polynomials
$E_0(T,1),E_1(T,1),\ldots,E_{(p-3)/2}(T,1)$ are monic of consecutive
degrees. Their evaluation matrix on the distinct elements of $\mathcal U$
is therefore a Vandermonde matrix multiplied by an invertible unitriangular
matrix. It is nonsingular, so all $u_t$ vanish. This proves uniqueness.

It remains to verify \eqref{eq:claimed-trace-residues}. Put
\[
        \gamma_t:=
        \begin{cases}
        1,&t^2-4\text{ is a square or zero},\\
        t^2-3,&t^2-4\text{ is a nonsquare},
        \end{cases}
\]
and
\[
        P_d(\lambda):=E_d(\lambda+\lambda^{-1},1)
        =\lambda^d+\lambda^{d-2}+\cdots+\lambda^{-d}.
\]
Let $\Omega=\{\lambda\in\F_{p^2}^\ast:\lambda^{p+1}=1\}$ 
be the norm-one subgroup. Regular split traces are parametrised twice by
$t=\lambda+\lambda^{-1}$ with $\lambda\in\F_p^\ast$, and nonsplit traces
are parametrised twice by the same formula with
$\lambda\in\Omega\setminus\{\pm1\}$. The values $t=\pm2$ occur once in
each parametrisation. Moreover, at $\lambda=\pm1$,
\[
        \lambda^2+\lambda^{-2}-1=1,
\]
so the two half-contributions at $t=\pm2$, where the two
preimages $\lambda$ and $\lambda^{-1}$ coincide, combine to give the
required single contribution. Since, for a nonsplit trace,
$t^2-3=\lambda^2+\lambda^{-2}-1$, we obtain
\begin{equation}
\sum_{t\in\F_p}\gamma_tE_d(t,1)
=
\frac12\left(
\sum_{\lambda\in\F_p^\ast}P_d(\lambda)
+
\sum_{\lambda\in\Omega}
(\lambda^2+\lambda^{-2}-1)P_d(\lambda)
\right).
\label{eq:split-nonsplit-parametrisation}
\end{equation}

For a finite cyclic group $H$,
\[
        \sum_{\lambda\in H}\lambda^m
        =
        \begin{cases}
        |H|,&|H|\mid m,\\
        0,&|H|\nmid m.
        \end{cases}
\]
The exponents occurring in $P_d$ lie between $-d$ and $d$. Since
$0\le d\le p-2$, the only exponent divisible by $p-1$ is $0$. Therefore
\begin{equation}
        \sum_{\lambda\in\F_p^\ast}P_d(\lambda)
        =
        \begin{cases}
        -1,&d\text{ even},\\
        0,&d\text{ odd},
        \end{cases}
        \quad\text{in }\F_p.
        \label{eq:split-geometric-sum}
\end{equation}

The exponents in
$(\lambda^2+\lambda^{-2}-1)P_d(\lambda)$ lie between $-d-2$ and $d+2$,
and hence between $-p$ and $p$. Thus the only exponent divisible by $p+1$
is again $0$. Since $|\Omega|=p+1\equiv1\pmod p$, the sum over $\Omega$
is the coefficient of $\lambda^0$, namely
\begin{equation}
        \sum_{\lambda\in\Omega}
        (\lambda^2+\lambda^{-2}-1)P_d(\lambda)
        =
        \begin{cases}
        -1,&d=0,\\
        1,&d\ge2\text{ even},\\
        0,&d\text{ odd}.
        \end{cases}
        \label{eq:nonsplit-geometric-sum}
\end{equation}
For even $d\ge2$, the coefficient is $1+1-1=1$, coming respectively from
the $\lambda^{-2}$, $\lambda^2$, and $\lambda^0$ terms of $P_d$.
Combining \eqref{eq:split-nonsplit-parametrisation},
\eqref{eq:split-geometric-sum}, and \eqref{eq:nonsplit-geometric-sum} gives
\[
        \sum_{t\in\F_p}\gamma_tE_d(t,1)
        =
        \begin{cases}
        -1,&d=0,\\
        0,&1\le d\le p-2,
        \end{cases}
\]
so \eqref{eq:claimed-trace-residues} is indeed the unique solution. Since a
nonsplit trace value is never $\pm2$, we have $c_t(x)=C_t(x)$ there, and
hence
\[
        c_t(x)\equiv t^2-3\pmod p.\qedhere
\]
\end{proof}

\section{The proof of Theorem \ref{maintheorem}}

The following elementary lemma explains why the threshold ``$p\ge13$''
turns up.

\begin{Lemma}
\label{lem:choose-good-Q}
For every prime $p\ge13$, there is an integer $Q$ with $4\le Q\le p-6$
such that $Q+3$ is a square and $Q-1$ is a nonsquare in $\F_p$.
\end{Lemma}

\begin{proof}
Put $z=Q+3$. It suffices to find a square $z\in\F_p$ such that $z-4$ is a
nonsquare and whose standard integer representative lies in
$\{7,8,\ldots,p-3\}$. The number of $z\in\F_p\setminus\{0,4\}$ satisfying
$\chi(z)=1$ and $\chi(z-4)=-1$ is
\[
\frac14\sum_{z\neq0,4}(1+\chi(z))(1-\chi(z-4)).
\]
This is equal to
\[
\frac14\left( p-2+\sum_{z\neq0,4}\chi(z)
 -\sum_{z\neq0,4}\chi(z-4)
 -\sum_{z\neq0,4}\chi(z(z-4))
\right)=\frac{p-2+\chi(-1)}4.
\]
Here the last equality uses Lemma~\ref{lem:quadratic-character-sum}, together
with
\[
        \sum_{z\neq0,4}\chi(z)=-1,
        \qquad
        \sum_{z\neq0,4}\chi(z-4)=-\chi(-1).
\]

For $p\ge37$, this number is greater than $7$. Outside the required
interval, and apart from $0$ and $4$, there are at most the seven values
$1,2,3,5,6,-2,-1$. Hence an admissible $z$ exists. For the remaining
primes, one may take
\[
\begin{array}{c|cccccc}
 p&13&17&19&23&29&31\\
\hline
 Q&6&6&4&6&4&4.
\end{array}
\]
In each case, $Q+3$ is a square and $Q-1$ is a nonsquare.
\end{proof}

We now prove Theorem \ref{maintheorem}.

\begin{proof}[Proof of Theorem \ref{maintheorem}]
Let $p$ be a prime with $p\ge13$. Suppose, for a contradiction, that such a
subset $S\subseteq G$ exists, and put $n:=|S|=p^2-1$. Choose $Q$ as in
Lemma~\ref{lem:choose-good-Q}, and choose $a\in\F_p$ with $a^2=Q+3$.
Then $a^2-4=Q-1$ is a nonsquare. Put
$b:=a^2-2=Q+1$. Since $b^2-4=a^2(a^2-4)$, the value $b$ is also nonsplit. Let 
$R\in\{0,1,\ldots,p-1\}$ be the standard integer representative of
$b^2-3$. Since $a^2-3=Q$, Theorem~\ref{thm:nonsplit-residue-general}
gives
\[
        c_a(x)\equiv Q\pmod p,
        \qquad
        c_b(x)\equiv R\pmod p
        \qquad(x\in G).
\]

Define
\[
        F_Q(z):=(z-Q)(z-Q-p).
\]
Since $Q$ and $R$ are the standard representatives of their residue
classes, the preceding congruences imply
$F_Q(c_a(x))\ge0$ for all $x\in G$, and $c_b(x)-R\ge0$ for all $x\in S$.
Consequently,
\begin{equation}
        0\le
        T:=\sum_{x\in G}F_Q(c_a(x))
        +p\sum_{x\in S}\bigl(c_b(x)-R\bigr).
        \label{eq:alternative-nonnegative-T}
\end{equation}

By Lemma~\ref{lem:intersection-numbers}, the trace relation $R_a$ has
valency $p(p-1)$, and hence
\[
        \sum_{x\in G}c_a(x)=p(p-1)n.
\]
Also $|G|=pn$. Expanding \eqref{eq:alternative-nonnegative-T} and applying
Proposition~\ref{prop:two-step-walks} gives
\[
\begin{aligned}
T
 &=p(p^2+p-2)n
   -(2Q+p)p(p-1)n
   +Q(Q+p)pn
   -pRn\\
 &=-pn\bigl(Q(p-2-Q)+R-2p+2\bigr).
\end{aligned}
\]
But
\[
        Q(p-2-Q)-4(p-6)
        =(Q-4)(p-6-Q)\ge0,
\]
and therefore
\[
        Q(p-2-Q)+R-2p+2
        \ge4(p-6)-2p+2
        =2p-22
        \ge4.
\]
Thus $T<0$, contradicting \eqref{eq:alternative-nonnegative-T}.
\end{proof}

\begin{Remark}
    We have shown that a sharply transitive subset of $\SL(2,q)$ exists if and only if
    $q\in\{2,3,5,7,11\}$. Indeed, it was shown by computer in \cite[\S5]{CoolsaetDeBeuleSiciliano}
    that any sharply transitive subset, in these cases, is a coset of a subgroup of $\SL(2,q)$. 
\end{Remark}

\subsection*{Acknowledgements}

The second author was supported by postdoctoral fellowship 1267923N from the Research Foundation Flanders (FWO),
and is extremely grateful to The University of Western Australia for their hospitality during his
visit in November, 2024.

\end{document}